\documentclass[12pt,twoside]{amsart}
\usepackage[all]{xy}
\usepackage{amssymb,amsmath,amsthm, amscd, enumerate, mathrsfs}
\usepackage{graphicx, hhline}
\usepackage{color}

\title{Fujita-type freeness for quasi-log canonical threefolds}
\author{Haidong Liu} 
\date{2019/2/18, version 0.05}
\address{Department of Mathematics, Graduate School of Science, 
Kyoto University, Kyoto 606-8502, Japan}
\email{jiuguiaqi@gmail.com}
\DeclareMathOperator{\Supp}{Supp}

\DeclareMathOperator{\mult}{mult}
\DeclareMathOperator{\Nqlc}{Nqlc}

\DeclareMathOperator{\Nqklt}{Nqklt}
\DeclareMathOperator{\Nklt}{Nklt}

\DeclareMathOperator{\Sing}{Sing}
\DeclareMathOperator{\Spec}{Spec}

\DeclareMathOperator{\ord}{ord}

\DeclareMathOperator{\Index}{Index}

\DeclareMathOperator{\red}{red}

\newtheorem{thm}{Theorem}[section]
\newtheorem{lem}[thm]{Lemma}
\newtheorem{prop}[thm]{Proposition}
\newtheorem{conj}[thm]{Conjecture}
\newtheorem{cor}[thm]{Corollary}

\theoremstyle{definition}

\newtheorem{defn}[thm]{Definition}
\newtheorem{rem}[thm]{Remark}
\newtheorem*{ack}{Acknowledgments}     
\newtheorem*{note}{Notation}

\newtheorem{case}{Case}

\makeatletter
    
    \@addtoreset{equation}{section}
\makeatother

\begin{document}

\maketitle 

\begin{abstract}
In this paper, 
we show that Fujita's basepoint-freeness conjecture 
for projective quasi-log canonical singularities 
holds true in dimension three.
Immediately, we prove Fujita-type basepoint-freeness for 
projective semi-log canonical threefolds.
\end{abstract}

\tableofcontents

\section{Introduction}\label{sec1}

This paper is a continuous study of \cite{fujino-haidong-freeness} to treat
Fujita's basepoint-freeness conjecture for quasi-log canonical singularities 
(see Section \ref{sec2} for a quick review of the theory of quasi-log schemes).
Note that we have posted Fujita's conjecture for quasi-log canonical pairs 
in \cite{fujino-haidong-freeness}.
Before we recall it in this paper, we first agree on a convention for reader's convenience.

\begin{note}\label{not1.1}
Let $N$ be a $\mathbb{R}$-Cartier divisor on a scheme $X$ of dimension $n$.
We call that $N$ satisfies {\em{Fujita's condition with respect to $n$}} 
if  
\begin{itemize}
\item[$(1)$]  $N^{\dim X_i}\cdot X_i>(\dim X_i)^{\dim X_i}$ 
for every positive-dimensional irreducible component $X_i$ of $X$, and
\item[$(2)$] for every positive $k$-dimensional irreducible
subvariety $Z$ which is not an irreducible 
component of $X$, 
we put 
$$n_Z=\min_i\{ \dim X_i \, |\, 
{\text{$X_i$ is an irreducible component of $X$ with 
$Z\subset X_i$}}\}$$ and assume 
that $N^{k}\cdot Z\geq n_Z^{k}$. 
\end{itemize}
When $X$ is equidimensional, our notation here is the same as that in \cite{fujino-slc-surface}
for semi-log canonical pairs, which is slightly stronger than that in
Fujita's original 
basepoint-freeness conjecture in \cite{fujita}.
\end{note}

We restate \cite[Conjecture 1.2]{fujino-haidong-freeness} by using this notation.

\begin{conj}[Fujita-type freeness for quasi-log canonical 
pairs]\label{conj-fujita}
Let $[X, \omega]$ be a projective quasi-log canonical 
pair of dimension $n$. 
Let $M$ be a Cartier divisor on $X$. We put $N=M-\omega$ 
and assume that $N$ satisfies Fujita's condition. 
Then the complete linear system $|M|$ is basepoint-free. 
\end{conj}

Osamu Fujino and the author have proved that Conjecture \ref{conj-fujita} holds true for $n\leq 2$
in \cite{fujino-haidong-freeness}. In this paper, we continue their work to prove that
Conjecture \ref{conj-fujita} holds true for $n=3$. That is:

\begin{thm}[Theorem \ref{thm-general}]\label{thm-main}
Let $[X, \omega]$ be a projective quasi-log canonical pair of dimension three.
Let $M$ be a Cartier divisor on $X$. We put $N=M-\omega$ 
and assume that $N$ satisfies Fujita's condition. 
Then the complete linear system $|M|$ is basepoint-free.
\end{thm}

In particular, since semi-log canonical pairs contain natural quasi-log canonical structures 
by \cite{fujino-funda-slc},
we have the following corollary immediately by Theorem \ref{thm-main}:

\begin{cor}[]\label{cor-slc}
Let $(X, \Delta)$ be a projective semi-log canonical pair of dimension three. 
Let $M$ be a Cartier divisor on $X$. 
We put $N=M-K_X-\Delta$
and assume that $N$ satisfies Fujita's condition.
Then the complete linear system $|M|$ is basepoint-free. 
\end{cor}

The basic idea to tackle basepoint-freeness conjecture
for a qlc pair $[X, \omega]$
is to use Riemann-Roch theorem and Fujita's condition to create an effective
divisor $L_1$ such that $L_1$ is very singular at point $x$ but as smooth as 
possible elsewhere. Then there is a maximal number $r_1<1$ such that  $[X, \omega+r_1L_1]$
is qlc at around $x$ and a minimal qlc center $Z$ passing through $x$.
Since $\dim Z < \dim X$ and  $[Z, (\omega+r_1L_1)|_{Z}]$ is an induced quasi-log pair which
is qlc at around $x$,
we can use  Riemann-Roch theorem and Fujita's condition repeatedly and finally get 
a sequence of numbers $c=\Sigma c_i$ and a sequence of divisors $L=\Sigma r_iL_i$ such that 
$x$ is the minimal qlc center of $[X, \omega+L]$. If $c<1$, then Fujita's conjecture is true
by the vanishing theorem. We call above operation {\em{inductive procedure}}.

\medskip

Let us quickly explain the strategy to prove Theorem \ref{thm-main} based on 
inductive procedure. 
Note that in the content of this paper, we will show this
strategy from the bottom up.
We take an arbitrary closed point $x$ of $X$.

\medskip

$\bullet$ If $x \in\Nqklt(X, \omega)$, then there exists an irreducible minimal qlc center 
$W$ passing through $x$ such that $\dim W<\dim X$.
By adjunction (see Theorem \ref{thm-av} (i)), $[W, \omega|_{W}]$ is a quasi-log canonical pair.
By the vanishing theorem (see Theorem \ref{thm-av} (ii)), 
the natural restriction map 
$H^0(X, \mathcal O_X(M))\to H^0(W, \mathcal O_{W}(M))$ is 
surjective. Therefore, we can replace $X$ with $W$ and use induction on the dimension. 

\medskip

$\bullet$ If $x \not\in\Nqklt(X, \omega)$, then $X$ is normal at $x$. 
Let $\nu: \widetilde X\to X$ be the normalization. 
Then by Theorem \ref{lem-normal}, $[\widetilde X, \nu^*\omega]$ 
is a qlc pair and isomorphic to $[X, \omega]$
in a neighborhood of $x$. To prove that $|M|$ is basepoint-free at $x$,
we try to prove that
$|\nu^*M|$ is basepoint-free at $\widetilde x:= \nu^{-1}(x)$.
To descend the obtained section back to $X$, we need more freeness for $|\nu^*M|$.
That is, we need that $\nu^*M$ can separated $\widetilde x$ and $\Nqklt(\widetilde X, \nu^*\omega)$.

\medskip

$\bullet$ We have deduced to show the (stronger) freeness under the assumption that $X$ is normal.
By using Theorem \ref{u-thm}, we can take a 
boundary $\mathbb R$-divisor $\Delta$ on 
$X$ such that $K_X+\Delta\sim _{\mathbb R} \omega+\varepsilon N$ for 
$0<\varepsilon \ll 1$ and  $(X, \Delta)$ is 
klt in a neighborhood of $x$. That is, by a small perturbation, we can ``almost''
view $[X,\omega]$ as a klt pair from now on, although
such a perturbation will weaken Fujita's condition a little bit.
Thanks to the first condition in Fujita's condition,
we can set the volume of $N$ bigger to offset the negative effect of this  perturbation.
Note that if \cite[Conjecture 1.5]{fujino-haidong} is true, then 
we don't need such a perturbation.

\medskip

$\bullet$ Next we assume that $x$ is not a terminal point on $X$.
Let $h:\widetilde X \to X$ be the terminalization by Lemma \ref{term-lem}
and $[\widetilde X, h^*\omega]$ be the induced qlc pair.
Then $\dim h^{-1}(x) \geq 1$ and thus we can choose a general point 
$\widetilde x \in h^{-1}(x)$ smooth on $\widetilde X$.
To prove that $|M|$ is basepoint-free at $x$ (and separate $\Nqklt( X, \omega)$),
we try to prove that
$|h^*M|$ is basepoint-free at $\widetilde x$
(and separate $\Nqklt(\widetilde X, h^*\omega)$).
Note that at this point, we get a qlc pair $[\widetilde X, h^*\omega]$ whose singularities are better,
but we lose Fujita's condition partially and ampleness on it.
Therefore, we need to go up to $\widetilde X$ to seach for high multiplicity (such that
$c=\Sigma c_i<1$ in the inductive procedure), 
and go down to $X$ to use ampleness and the vanishing theorem.

\medskip

$\bullet$ Finally we turn to consider that $x$ is a terminal point on $X$.
Let $p:X' \to X$ be the global index one cover defined in Section \ref{sec4}.
At this point, we only have a neighborhood $X'_0$ of point $x'=p^{-1}(x)_{\red}$
such that $(X'_0, \Delta'_0)$ is a klt pair (where $\Delta'_0=p^*\Delta|_{X'_0}$).
But Fujita's condition is kept for those possible qlc minimal centers passing through $x'$
coming from the inductive procedure.
Again, we go up to $X'$ to seach for high multiplicity, 
and go down to $X$ to use ampleness and the vanishing theorem.
\medskip

We strongly recommend those interested readers to read \cite{liu} and \cite{fujino-haidong-freeness}
as a warm-up on basepoint-freeness for quasi-log canonical singularities. 
Note also that Angehrn--Siu type effective freeness for 
quasi-log canonical pairs can be proved by above strategy without using inversion of adjunction 
for quasi-log canonical pairs in \cite[Theorem 2.10]{liu}.

\begin{ack}
The author would like to thank Professor Osamu Fujino all the time for his great 
support and useful suggestions. He would also like to thank Chen Jiang for many discussions on
Section \ref{sec4} and Kenta Hashizume for discussions on terminalization of generalized polarized 
pair.
\end{ack}

We will work over $\mathbb C$, the 
complex number field, throughout this paper. A scheme means
a separated scheme of finite 
type over $\mathbb C$. A variety means a reduced scheme, that is, a
reduced separated scheme of finite 
type over $\mathbb C$. 
We sometimes assume that a variety 
is irreducible 
without mentioning it explicitly if there is no risk of 
confusion. We will freely use the standard notation of 
the minimal model program and the 
theory of quasi-log schemes as in  \cite{fujino-funda} and  \cite{fujino-foundations}.
For the details of semi-log canonical pairs, see \cite{fujino-funda-slc}.

\section{Quasi-log schemes}\label{sec2}

In this section, we collect some basic definitions and 
explain some results on quasi-log schemes. 

\begin{defn}[$\mathbb R$-divisors]\label{u-def2.1}
Let $X$ be an equidimensional variety, 
which is not necessarily regular in codimension one. 
Let $D$ be an $\mathbb R$-divisor, 
that is, $D$ is a finite formal sum $\sum _i d_i D_i$, where 
$D_i$ is an irreducible reduced closed subscheme of $X$ of 
pure codimension one and $d_i$ is a real number for every $i$ 
such that $D_i\ne D_j$ for $i\ne j$. We put 
\begin{equation*}
D^{<1} =\sum _{d_i<1}d_iD_i, \quad 
D^{\leq 1}=\sum _{d_i\leq 1} d_i D_i, \quad 
D^{> 1}=\sum _{d_i>1} d_i D_i, \quad 
\text{and} \quad 
D^{=1}=\sum _{d_i=1}D_i. 
\end{equation*}
We also put 
$$
\lceil D\rceil =\sum _i \lceil d_i \rceil D_i \quad \text{and} 
\quad 
\lfloor D\rfloor=-\lceil -D\rceil, 
$$
where $\lceil d_i\rceil$ is the integer defined by $d_i\leq 
\lceil d_i\rceil <d_i+1$. 
When $D=D^{\leq 1}$ holds, 
we usually say that 
$D$ is a {\em{subboundary}} $\mathbb R$-divisor. 

Let $B_1$ and $B_2$ be $\mathbb R$-Cartier divisors on $X$. 
Then $B_1\sim _{\mathbb R} B_2$ means that 
$B_1$ is $\mathbb R$-linearly equivalent 
to $B_2$. 
\end{defn}

Let us quickly recall singularities of pairs for the reader's convenience. 
We recommend the reader to see \cite[Section 2.3]{fujino-foundations} 
for the details.  

\begin{defn}[Singularities of pairs]\label{u-def2.2} 
Let $X$ be a normal variety and let $\Delta$ be 
an $\mathbb R$-divisor on $X$ such that 
$K_X+\Delta$ is $\mathbb R$-Cartier. 
Let $f:Y\to X$ be a projective 
birational morphism from a smooth variety $Y$. 
Then we can write 
$$
K_Y=f^*(K_X+\Delta)+\sum _E a(E, X, \Delta)E, 
$$ 
where $a(E, X, \Delta)\in \mathbb R$ and $E$ is a prime 
divisor on $Y$. 
By taking $f:Y\to X$ suitably, 
we can define $a(E, X, \Delta)$ for any prime 
divisor $E$ {\em{over}} $X$ and call it 
the {\em{discrepancy}} of $E$ with respect to $(X, \Delta)$. 
If $a(E, X, \Delta)>-1$ (resp.~$a(E, X, \Delta)\geq -1$) holds 
for any prime divisor $E$ over $X$, then 
we say that $(X, \Delta)$ is {\em{sub klt}} (resp.~{\em{sub 
log canonical}}). 
If $(X, \Delta)$ is sub klt (resp.~sub log canonical) 
and $\Delta$ is effective, then we say that 
$(X, \Delta)$ is {\em{klt}} (resp.~{\em{log canonical}}). 
If $(X, \Delta)$ is log canonical 
and $a(E, X, \Delta)>-1$ for any prime divisor 
$E$ that is exceptional over $X$, 
then we say that $(X, \Delta)$ is {\em{plt}}. 

If there exist a projective birational morphism 
$f:Y\to X$ from a smooth 
variety $Y$ and a prime divisor $E$ on $Y$ such that 
$a(E, X, \Delta)=-1$ and $(X, \Delta)$ is log canonical 
in a neighborhood of the generic point of $f(E)$, 
then $f(E)$ is called a {\em{log canonical center}} of 
$(X, \Delta)$.  
\end{defn}

\begin{defn}[Multiplier ideal sheaves]\label{u-def2.3}
Let $X$ be a normal variety and let $\Delta$ 
be an effective $\mathbb R$-divisor 
on $X$ such that $K_X+\Delta$ is $\mathbb R$-Cartier. 
Let $f:Y\to X$ be a projective 
birational morphism 
from a smooth variety 
such that 
$$
K_Y+\Delta_Y=f^*(K_X+\Delta)
$$ 
and $\Supp \Delta_Y$ is a simple normal crossing 
divisor on $Y$. 
We put 
$$
\mathcal J(X, \Delta)=f_*\mathcal O_Y(-\lfloor \Delta_Y\rfloor)
$$ 
and call it the {\em{multiplier ideal 
sheaf}} of $(X, \Delta)$. 
We can easily check that 
$\mathcal J(X, \Delta)$ is a well-defined 
ideal sheaf on $X$. 
The closed subscheme defined by 
$\mathcal J(X, \Delta)$ is denoted by 
$\Nklt(X, \Delta)$. 
\end{defn}

The notion of {\em{globally embedded simple normal crossing 
pairs}} plays a crucial role in the theory of quasi-log schemes 
described in \cite[Chapter 6]{fujino-foundations}. 

\begin{defn}[Globally embedded simple normal 
crossing pairs]\label{u-def2.4} 
Let $Y$ be a simple normal crossing 
divisor on a smooth variety $M$ and let $B$ be 
an $\mathbb R$-divisor 
on $M$ such that 
$Y$ and $B$ have no common irreducible components and 
that the support of $Y+B$ is a simple normal crossing divisor on $M$. In this 
situation, $(Y, B_Y)$, where $B_Y:=B|_Y$, 
is called a {\em{globally embedded simple 
normal crossing pair}}. 
A {\em{stratum}} of $(Y, B_Y)$ means a log canonical 
center of $(M, Y+B)$ included in $Y$.  
\end{defn}

Let us recall the notion of {\em{quasi-log schemes}}, 
which was first introduced by Florin Ambro (see \cite{ambro}). 
The following definition is slightly different from 
the original one. 
For the details, see \cite[Appendix A]{fujino-pull-back}. 
In this paper, 
we will use the framework of quasi-log schemes 
established in 
\cite[Chapter 6]{fujino-foundations}. 

\begin{defn}[Quasi-log schemes]\label{u-def2.5} 
A {\em{quasi-log scheme}} is a scheme $X$ endowed with 
an $\mathbb R$-Cartier divisor (or $\mathbb R$-line bundle) $\omega$ on $X$, 
a closed subscheme $\Nqlc(X, \omega)\subsetneq X$,  
and a finite collection $\{C\}$ of reduced and irreducible subschemes 
of $X$ such that there exists a proper morphism 
$f:(Y, B_Y)\to X$ from a globally embedded simple normal 
crossing pair $(Y, B_Y)$ satisfying 
the following properties: 
\begin{itemize}
\item[(1)] $f^*\omega\sim_{\mathbb R} K_Y+B_Y$. 
\item[(2)] The natural map 
$\mathcal O_X\to f_*\mathcal O_Y(\lceil -(B^{<1}_Y)\rceil)$ induces 
an isomorphism 
$$
\mathcal I_{\Nqlc(X, \omega)}\overset{\sim}\longrightarrow
f_*\mathcal O_Y(\lceil -(B^{<1}_Y)\rceil-\lfloor B^{>1}_Y\rfloor), 
$$ 
where $\mathcal I_{\Nqlc(X, \omega)}$ is the defining ideal sheaf of 
$\Nqlc(X, \omega)$. 
\item[(3)] The collection of subvarieties $\{C\}$ 
coincides with the images of $(Y, B_Y)$-strata that are 
not included in $\Nqlc(X, \omega)$.  
\end{itemize}
We simply write $[X, \omega]$ to denote 
the above data 
$$\left(X, \omega, f:(Y, B_Y)\to X\right)
$$ 
if there is no risk of confusion. 
We note that the subvarieties $C$ are called the {\em{qlc strata}} 
of $\left(X, \omega, f:(Y, B_Y)\to X\right)$ or simply of $[X, \omega]$. 
If $C$ is a qlc stratum of $[X, \omega]$ but is not an irreducible 
component of $X$, then $C$ is called a {\em{qlc center}} 
of $[X, \omega]$. 
The union of all qlc centers of $[X, \omega]$ is denoted by 
$\Nqklt(X, \omega)$. 
\end{defn}

If $B_Y$ is a subboundary $\mathbb R$-divisor, 
then $[X, \omega]$ in Definition \ref{u-def2.5} is 
called a {\em{quasi-log canonical pair}}. 

\begin{defn}[Quasi-log canonical pairs]\label{u-def2.6}
Let $(X, \omega, f:(Y, B_Y)\to X)$ be a quasi-log scheme as 
in Definition \ref{u-def2.5}. 
We say that $(X, \omega, f:(Y, B_Y)\to X)$ or simply $[X, \omega]$ 
is a {\em{quasi-log canonical pair}} ({\em{qlc pair}}, for short) 
if $\Nqlc(X, \omega)=\emptyset$. 
Note that the condition $\Nqlc(X, \omega)=\emptyset$ is equivalent 
to $B^{>1}_Y=0$, that is, $B_Y=B^{\leq 1}_Y$. 
\end{defn}

One of the most important results in the theory of quasi-log schemes 
is the following theorem. 

\begin{thm}\label{thm-av} 
Let $[X,\omega]$ be a 
quasi-log scheme and let $X'$ be the union of $\Nqlc(X, \omega)$ with a 
$($possibly empty$)$ union 
of some qlc strata of $[X,\omega]$. 
Then we have the following properties.
\begin{itemize}
\item[(i)] {\em{(Adjunction)}}.~Assume 
that $X'\neq \Nqlc(X, \omega)$. 
Then $[X', \omega']$ is a quasi-log scheme with $\omega'=\omega \vert_{X'}$ 
and $\Nqlc(X', \omega')=\Nqlc(X, \omega)$. Moreover, the qlc 
strata of $[X',\omega']$ are exactly the qlc strata 
of $[X,\omega]$ that are included in $X'$.
\item[(ii)] {\em{(Vanishing theorem)}}.~Assume 
that $ \pi: X \rightarrow S$ is a proper 
morphism between schemes. 
Let $L$ be a Cartier divisor on $X$ such that 
$L-\omega$ is nef and log big over $S$ 
with respect to $[X,\omega]$, that is, 
$L-\omega$ is $\pi$-nef and 
$(L-\omega)|_C$ is $\pi$-big for every qlc stratum $C$ of $[X, \omega]$. 
Then $R^{i}\pi_{*}(\mathcal {I}_{X'}\otimes \mathcal{O}_{X}(L))=0$ 
for every $i>0$, where $\mathcal{I}_{X'}$ is the 
defining ideal sheaf of $X'$ on $X$.
\end{itemize}
\end{thm}

For the proof of Theorem \ref{thm-av}, see, 
for example, \cite[Theorem 6.3.5]{fujino-foundations}. 
We note that we generalized Koll\'ar's torsion-free and 
vanishing theorems in \cite[Chapter 5]{fujino-foundations} 
by using the theory of mixed Hodge structures on cohomology with 
compact support 
in order to establish Theorem \ref{thm-av}. 

\medskip

The following theorem is a special case of \cite[Theorem 1.5]
{fujino-slc-trivial}. It is a deep result 
based on the theory of variations of mixed Hodge structures 
on cohomology with compact support. 

\begin{thm}[{\cite[Theorem 1.5]{fujino-slc-trivial}}]\label{u-thm} 
Let $[X, \omega]$ be a quasi-log canonical pair such that 
$X$ is a normal projective irreducible variety. 
Then there exists a projective 
birational morphism 
$p:X'\to X$ from a smooth projective 
variety $X'$ such that 
$$
K_{X'}+B_{X'}+M_{X'}=p^*\omega, 
$$ 
where $B_{X'}$ is a subboundary $\mathbb R$-divisor, that is, 
$B_{X'}=B^{\leq 1}_{X'}$, 
such that $\Supp B_{X'}$ is a simple normal crossing 
divisor and that $p_*B_{X'}$ is effective,  
and $M_{X'}$ is a nef $\mathbb R$-divisor on $X'$. 
Furthermore, we can make $B_{X'}$ satisfy 
$p(B^{=1}_{X'})=\Nqklt(X, \omega)$.  
\end{thm}

Next theorem will play an important role as we explained in
the strategy of introduction.

\begin{thm}[{\cite[Theorem 1.1]{fujino-haidong}}]\label{lem-normal} 
Let $[X, \omega]$ be a quasi-log canonical pair such that 
$X$ is irreducible. 
Let $\nu: \widetilde{X}\to X$ be the normalization. 
Then $[\widetilde{X}, \nu^*\omega]$ naturally becomes a 
quasi-log canonical pair with 
the following properties: 
\begin{itemize}
\item[(i)] if $C$ is a qlc center of $[\widetilde{X}, \nu^*\omega]$, 
then $\nu(C)$ is a qlc center of $[X, \omega]$, 
and 
\item[(ii)] $\Nqklt(\widetilde{X}, \nu^*\omega)=\nu^{-1}(\Nqklt(X, \omega))$. 
More precisely, the equality 
$$
\nu_*\mathcal I_{\Nqklt(\widetilde{X}, \nu^*\omega)}=
\mathcal I_{\Nqklt(X, \omega)}
$$ 
holds, where 
$\mathcal I_{\Nqklt(X, \omega)}$ 
and $\mathcal I_{\Nqklt(\widetilde{X}, \nu^*\omega)}$ are the defining ideal sheaves of 
$\Nqklt(X, \omega)$ and $\Nqklt(\widetilde{X}, \nu^*\omega)$ respectively. 
\end{itemize}  
\end{thm}

We prepare one more useful lemma for our paper.
It is a kind of 
terminalization, which is well known to experts in the framework of log canonical pairs.

\begin{lem}[terminalization]\label{term-lem}
Let $\left(X, \omega, f:(Y, B_Y)\to X\right)$ be a quasi-log canonical 
pair such that 
$X$ is a normal irreducible variety. 
Then there is a morphism $h: \widetilde{X} \to X$ where $\widetilde{X}$ is a normal $\mathbb Q$-factorial
 terminal variety and an induced quasi-log canonical 
pair $\left(\widetilde{X}, \widetilde{\omega}, 
\widetilde{f}:(\widetilde{Y}, B_{\widetilde{Y}})\to \widetilde{X}\right)$  where 
$\widetilde{\omega}=h^* \omega$.
\end{lem}

\begin{proof}
By a modification of $f$ (cf. \cite{fujino-haidong}), we can assume that
every stratum of $Y$ is dominant onto $X$. By \cite[Lemma 11.1]{fujino-slc-trivial},
we can decompose $\left(X, \omega, f:(Y, B_Y)\to X\right)$ 
into a combination of qlc $\mathbb Q$-structures
and prove our lemma for each qlc $\mathbb Q$-structure and then for the final $\mathbb R$-structure
by combination as the proof of \cite[Theorem 1.7]{fujino-slc-trivial}.
Therefore, we can assume that $\left(X, \omega, f:(Y, B_Y)\to X\right)$
has a $\mathbb Q$-structure and thus 
is a so-called {\em{basic slc-trivial fibration}} \cite[Definition 4.1]{fujino-slc-trivial}.
By Theorem \ref{u-thm} and further blowing ups of $(Y, B_Y)$, there is a 
commutative diagram as follows:
$$
\xymatrix{
& (Y, B_Y)\ar[ld]_-{g}\ar[d]^-f\\ 
X' \ar[r]_-{p}& X
}
$$
such that $K_{X'}+B^{+}_{X'}+M_{X'}\sim_{p, \mathbb Q}B^{-}_{X'}$ 
where $B_{X'}=B^{+}_{X'}-B^{-}_{X'}$, $B^{+}_{X'}$ and $B^{-}_{X'}$ are effective $\mathbb Q$-divisors.
Note that $(X', B^{+}_{X'})$ is log smooth and $M_{X'}$ is a nef $\mathbb Q$-divisor on $X'$. 
We can run the relative $(K_{X'}+B^{+}_{X'}+M_{X'})$-MMP over $X$ and terminate at a map
$h: \widetilde{X} \to X$ such that $B^{-}_{X'}$ is contracted. Then
we get the following commutative diagram:
$$
\xymatrix{
& (Y, B_Y)\ar[ld]_-{g}\ar[d]^-{\widetilde{f}}\ar[rd]^-{f}& \\ 
X' \ar[r]_-{\pi}& \widetilde{X} \ar[r]_-{h}& X
}
$$
where $\widetilde{X}$ is $\mathbb Q$-factorial.
Let $K_{\widetilde{X}}=\pi_* K_{X'}$, $B_{\widetilde{X}}=\pi_* B^{+}_{X'}$
and  $M_{\widetilde{X}}=\pi_* M_{X'}$. Then $\widetilde{\omega}:=K_{\widetilde{X}}+B_{\widetilde{X}}+M_{\widetilde{X}}=h^*\omega$
and $X'\to \widetilde{X} \to X$ is a generalized polarized dlt pair such that $B_{\widetilde{X}}$
is the boundary part and $M_{\widetilde{X}}$ the nef part (cf. \cite[Definition 1.4]{bz}).
Replacing $\widetilde{X}$ with its terminalization, we can further assume that $\widetilde{X}$ is terminal.
By Definition \ref{u-def2.5}, 
to prove that  $\left(\widetilde{X}, \widetilde{\omega}, 
\widetilde{f}:(Y, B_Y) \to \widetilde{X}\right)$
 is a quasi-log canonical pair,
we only need to prove that
$$
\alpha:  \mathcal O_{\widetilde{X}} \hookrightarrow 
\widetilde{f}_*\mathcal O_{Y}(\lceil -B_{Y}^{<1}\rceil)
$$
is an isomorphism.
Since $h:\widetilde{X} \to X$ is an isomorphism at the generic point of any prime divisor
on $X$,
we only need to check that $\alpha$ is isomorphic at every generic point of exceptional locus.
Let $P$ be a prime divisor in the exceptional locus and $t_P=\mult_P  B_{\widetilde{X}}$.
Since $B_{\widetilde{X}}$ is a boundary , $0\leq t_p\leq 1$.
Note that $\left(\widetilde{X}, \widetilde{\omega}, 
\widetilde{f}:(Y, B_Y) \to \widetilde{X}\right)$ is an induced basic slc-trivial fibration.
In particular, $K_Y+B_Y+(1-t_P)\widetilde{f}^*P$ is sub slc over the generic point of $P$ by 
\cite[(4.5)]{fujino-slc-trivial}. That is, 
there is a prime divisor $F$ on $Y$ dominant onto $P$, $a:=\mult_F B_Y$ and $b:=\mult_F {f}^*P$
such that $a+(1-t_P)b=1$. 
Since $\widetilde X$ is smooth in codimension two,  
the generic point of $P$ is Cartier. That is, $b$ is an integer. 
Then 
$$
\lceil - a \rceil=\lceil(1-t_P)b-1\rceil \leq b-1.
$$
This is equivalent to say that
$\lceil -B_{Y}^{<1}\rceil \npreceq f^*P$. Therefore,
$$
 \mathcal O_{\widetilde{X}} \hookrightarrow 
\widetilde{f}_*\mathcal O_{Y}(\lceil -B_{Y}^{<1}\rceil) \subsetneq
 \mathcal O_{\widetilde{X}}(P)
$$
at the generic point of $P$. That is, $\mathcal O_{\widetilde{X}} \simeq
\widetilde{f}_*\mathcal O_{Y}(\lceil -B_{Y}^{<1}\rceil)$ and this is what we want.
\end{proof}

\begin{cor}\label{term-cor}
Let $[X, \omega]$ be a qlc pair such that $X$ is normal and 
$[\widetilde{X}, \widetilde{\omega}]$ be the terminalization of $[X, \omega]$ induced
by $h: \widetilde{X} \to X$ as in Lemma \ref{term-lem}.
Let $L$ be an effective $\mathbb R$-Cartier divisor on  $X$ and $\widetilde{L}=h^*L$.
Then $[X, \omega+L]$ is qlc at a point $x$
if and only if $[\widetilde{X}, \widetilde{\omega}+\widetilde{L} ]$
is qlc at every point of $h^{-1}(x)$.
In particular, let $W$ be a connected union of qlc centers of $[X, \omega+L]$
and $\widetilde{W}$ be the union of all qlc centers of  $[\widetilde{X}, \widetilde{\omega}+\widetilde{L}]$
mapping into $W$, then $\widetilde{W}$ is also connected.
\end{cor}

\begin{proof}
By the construction of the induced qlc pair $[\widetilde{X}, \widetilde{\omega}]$
as in Lemma \ref{term-lem}, $[\widetilde{X}, \widetilde{\omega}]$ and $[X, \omega]$
are dominated by the same globally embedded simple normal crossing pair $(Y, B_Y)$.
Therefore, the first part is a direct implication of Definition \ref{u-def2.6}.
By adjunction theorem, $W$ and $\widetilde{W}$ are also dominated by the same union of 
strata $Y'$ of $(Y, B_Y)$. Since $W$ is connected, 
$Y'$ is also connected and thus so is $\widetilde{W}$.
\end{proof}

\section{Deficit at a point}\label{sec3}

In this section, we collect some definitions and 
explanations to reach to a definition of {\em{deficit}} of a klt pair at a closed point,
which is a number to measure how far this pair away from having
this point as its minimal lc center.
Most of them are due to Ein's \cite{ein}, Ein-Lazarsfeld's \cite{el},
Lee's \cite{lee} and Helmke's \cite{h1}, \cite{h2}, and the references therein.

\begin{defn}\label{defn-ord}
Let $X$ be a variety with $\dim X=n$ and $x$ be a closed point on $X$.
Let $\varphi: X' \to X$ be the blowing up of point $x$ 
(blowing up with respect to $m_x$ where $m_x$ is the maximal ideal sheaf of point $x$) 
and $E=\varphi^{-1}(x)$ be the Cartier exceptional divisor.
Note that $E$ is not necessarily reduced or irreducible.
Let $G$ be an effective $\mathbb R$-Cartier divisor on $X$. We define 
{\em{the order of $G$ at point $x$}},
 {\em{$\ord_x G$}} for short, to be the coefficient of $E$ in $\varphi^*G$.
\end{defn}


\begin{rem}\label{rem-ord}
When $G$ is Cartier, we can define 
$$
\ord_x G=\ord_x(f)=\max\{n\in\mathbb N;f\in m_x^n\}
$$
where $f$ is a local equation of $G$ and $m_x$ is the maximal ideal sheaf of point $x$.
We can also generalize it to effective $\mathbb R$-Cartier divisors by $\mathbb R$-linear combination. 
Note that these two definitions of $\ord_x G$ are the same.
Note also that if $x$ is smooth and $G$ is prime, then $\ord_x G=\mult_x G$
where $\mult_x G$ is denoted as the multiplicity of a variety at $x$.
In general, $\ord_x G$ and $\mult_x G$ are not the same, which will make some confusions.
Therefore, we will use $\ord_x G$ (order) for divisors and $\mult_x G$ (multiplicity) for varieties. 
\end{rem}

Let $(X,\Delta)$ be a log pair which is klt at around a closed point $x$.
Let $G$ be an effective $\mathbb R$-Cartier divisor.
Let $f:Y\to X$ be a log resolution of $(X,\Delta+G)$
that factors through the blowing up $\varphi$ by 
a morphism $g:Y\to X'$ such that $f=\varphi\circ g$. 
Let $K_Y+B_Y=f^*(K_X+\Delta)$.

\begin{defn}[Deficit]\label{defn-def}
Let $B_Y=\sum a_iF_i$, $f^*G=\sum b_iF_i$ and
$g^*E=\sum e_iF_i$ where $F_i$ are simple normal crossing divisors. 
Assume that $(X,\Delta+G)$ is log canonical at around $x$. Then:
\begin{itemize}
\item[$(1)$]
If $x \notin\Nklt(X,\Delta+G)$, then the {\em{deficit of $G$}} at $x$, is defined as 
$$
d_x(G)=\inf_{f(F_i)=x}\{\frac{1-a_i-b_i}{e_i}\},
$$
where $f$ varies among all those log resolutions
factoring through $\varphi$.
\item[$(2)$] If $x \in\Nklt(X,\Delta+G)$ but $x \notin\Nklt(X,\Delta+(1-t)G)$ for any $0<t<1$,
then the {\em{deficit of $G$}} at $x$, is  defined as 
$$
d_x(G)=\lim_{t\to 0^{+}} d_x((1-t)G).
$$
\end{itemize}
When $G=0$, we simply denote $d_x(G)$ as $d_x$ for short.
\end{defn}

\begin{rem}\label{rem-def-1}
In Definition \ref{defn-def} $(2)$, the minimal lc center
$W$ of  $(X,\Delta+G)$ passing through $x$ is always called the {\em{critical variety}} of $G$.
We also call that  $G$ is {\em{critical}} at $x$ if not necessary to mention $W$.
See \cite[Definition 2.4]{ein} or \cite[Definition 2.5]{lee}.
But we don't need these definitions in our paper since we don't need the tie-breaking trick 
(cf. \cite[Remark 2.5]{ein} or \cite[Remark 2.6]{lee}).
\end{rem}

\begin{rem}\label{rem-def-2}
It is equivalent to define  {\em{deficit of $G$}} as the smallest $c\in \mathbb R_{\geq 0}$ such that 
for any effective $\mathbb R$-Cartier divisor $D$ with $\ord_x D\geq c$, 
we will have $x\in\Nklt (X,\Delta+G+D)$. This is the approach used in \cite[Definition 2.9]{lee} 
when $x$ is a singular point on $X$ based on \cite[Section 4]{ein}.
When $x$ is smooth, our two definitions of deficit coincide with 
the so-called {\em{local discrepancy}} defined by Helmke \cite{h1} \cite{h2}.
\end{rem}

\begin{rem}\label{rem-def-3}
We could also define {\em{deficit}} of a qlc pair $[X,\omega]$ at a closed point $x$
where $x \notin \Nqklt(X,\omega)$.
But thanks to Theorem \ref{u-thm}, we could ``almost'' turn a qlc pair into a klt pair
by a small perturbation. See the Appendix in this paper.
\end{rem}

\begin{defn}[nice lifting]\label{defn-lift}
Notations are as in Definition \ref{defn-def} $(2)$.
Let $W$ be the minimal lc center of  $(X,\Delta+G)$ passing through $x$.
Let $D$ be an effective $\mathbb R$-Cartier divisor on $W$. 
An effective $\mathbb R$-Cartier
divisor $B$ on X is said to be a {\em{nice lifting}} of $D$, if $B$ satisfies the following two properties:
	
\begin{itemize}
\item[(1)]~$B|_W=D$;
\item[(2)]~$\Nklt(X-W, (\Delta+G+B)|_{X-W})=\Nklt(X-W,  (\Delta+G)|_{X-W})$.
\end{itemize}
\end{defn}

\begin{rem}\label{rem-lift}
Roughly speaking, a lifting $B$ of $D$ is {\em{nice}}, means that $B$ keeps the required singularities 
as $D$ on $W$ but as {\em{smooth}} as possible outside of $W$. Up to quasi-log canonical
singularities, a nice lifting of $D$ 
will always exist and if there are two nice liftings of $D$, they will contribute nothing difference
in our settings. Therefore, we can choose any nice lifting as we want. See \cite[Definition 4.5]{ein},
\cite[Proposition 2.7]{lee}, \cite[Claim 6.8.4]{ko}, \cite[Lemma 2.9]{fujino-as} 
and \cite[Proposition 3.4]{liu} up to quasi-log 
canonical singularities for more details of nice lifting.
\end{rem}

By above definitions, we immediately have the following propositions:

\begin{prop}[]\label{prop-def}
Let $(X,\Delta)$ be a klt pair. Let $G$ be an effective $\mathbb R$-Cartier divisor
and $x$ be a closed point on $X$.
Let $W$ be the minimal lc stratum of  $(X,\Delta+G)$ passing through $x$ 
with $\dim W=k$. Let $D$ be an effective $\mathbb R$-divisor on $W$.
Then:
\begin{itemize}
\item[$(1)$]~$d_x(G)\leq k$. If $W$ is singular at $x$, then ~$d_x(G)\leq k-1$;
\item[$(2)$]~we can choose a nice lifting $B$ such that $\ord_x B\geq \ord_x D$;
\item[$(3)$]~$d_x((1-t)G+B)\leq d_x((1-t)G)-\ord_x B$ for sufficiently small $t$.
\end{itemize}
\end{prop}

\begin{proof}
See \cite[Proposition 4.2 and Lemma 4.6]{ein} or \cite[Proposition 2.9]{lee}. 
See also \cite[Equation (2.6)]{h2}.
\end{proof}

\section{Global index one cover}\label{sec4}
Recall that an {\em{index one cover}} (cf. \cite[(6.8)]{ckm} or \cite[Definition 5.19]{km}) 
is constructed locally on an affine variety.
In this section, we construct index one cover globally. 
Let $X$ be a projective $n$-dimensional normal $\mathbb Q$-Gorenstein variety.
We assume that $x:=\Sing X$ is a unique isolated point.
Let $r=\Index_xX$. 
That is, $r$ is the smallest positive integer such that $rK_X$ is Cartier. 
Take a sufficiently ample line bundle $\mathcal A$ on $X$ such that 
$\mathcal A^r \otimes \mathcal O_X(rK_X)$ is generated by global sections. 
Let $D=(s=0)$ be a general member of $|\mathcal A^r \otimes \mathcal O_X(rK_X)|$
which is smooth and does not pass through point $x$.
By general ramified cyclic cover in
\cite[Definition 2.52]{km},
there is a cyclic cover 
$$
p: X'=\Spec_X\oplus_{i=0}^{r-1} (\mathcal A\otimes \mathcal O_X(K_X))^{\lbrack -i\rbrack}\to X
$$ ramified along $x\cup D$ with $\deg p =r$.
In particular, let $U$ be an open neighborhood of point $x$  and $U'=p^{-1}(U)$. Then
$$
p|_{U'} : U'=\Spec_U\oplus_{i=0}^{r-1} (\mathcal A\otimes \mathcal O_X(K_X)|_U)^{\lbrack -i\rbrack}\to U
$$
is the restriction of $p$ and the multiplication is given by 
$$
s|_U: \mathcal O_U \to  (\mathcal A\otimes \mathcal O_X(K_X)|_U)^{\lbrack r\rbrack} \simeq \mathcal O_U
$$
by shrinking $U$ suitably. Note that this is exactly the definition of local index one cover defined by $s|_U$.

\begin{defn}[global index one cover]\label{defn-global}
Such a cyclic cover $p: X'\to X$  is called {\em{global index one cover}} at point $x$. 
Note that $p$ is heavily depended on $D$.
\end{defn}

\begin{prop}\label{index-prop}
Let $p: X'\to X$ be a global index one cover at point $x$ with $r=\Index_xX$. Then:
\begin{itemize}
\item[(1)] $X'$ is normal and Gorenstein, i.e., $K_{X'}$ is Cartier, 
\item[(2)] $x':=p^{-1}(x)_{\red}$ is the unique possible singular point,  
\item[(3)] $p$ is \'etale in codimenison one over $X \backslash D$,
\item[(4)] the extension of the function fields $\mathbb C(X')/ \mathbb C(X)$ 
is Galois and the Galois group $G \cong \mathbb Z/(r)$ acts on $X'$ over $X$, and
\item[(5)] $g \cdot m_{x'} \simeq m_{x'}$ where $g\in G$ is an action and $m_{x'}$ is the 
ideal sheaf of point $x'$.
\end{itemize}
\end{prop}
\begin{proof}
Note that outside the singular point $x$, $p$ coincides with the ramified cyclic cover for
line bundle case as in \cite[Definition 2.50]{km}.
Since $D$ is smooth, $X'\backslash p^{-1}(x)$ is smooth by \cite[Lemma 2.51]{km}.
Note also that outside the ramified locus $D$, $p$ coincides with the local index one cover determined by
the nowhere vanishing local section of $D$. Then it is easy to see that
$(1)$--$(4)$ are direct conclusions of \cite[4-5-1]{matsuki}.
For $(5)$, we write down the local expression of $m_{x'}$ as in \cite[4-5-1]{matsuki}:
$$
m_{x'}=m_x \oplus \{ \oplus^{r-1}_{i=1} R_i\cdot (\sqrt[r]{f})^i\}.
$$
Assume the action of $g$ is presented by a fixed primitive $r$-th root of unity $\zeta$, then
$$
g \cdot  m_{x'}=\zeta \cdot (m_x \oplus \{ \oplus^{r-1}_{i=1} R_i\cdot (\sqrt[r]{f})^i\})
=m_x  \oplus \{ \oplus^{r-1}_{i=1} R_i\cdot (\zeta\cdot \sqrt[r]{f})^i\}).
$$
It is easy to see that $g \cdot m_{x'} \subset m_{x'}$. Using a converse action $g^{-1}$,
we get what we want.
\end{proof}

Let $(X, \Delta)$ be a projective klt pair such that 
$X$ is an $n$-dimensional normal $\mathbb Q$-Gorenstein variety. 
Assume that $x:=\Sing X$ is a closed point.
Let $r=\Index_xX$. 
Let $N$ be an ample $\mathbb R$-Cartier divisor on $X$ such that $N^n>\frac{n^n}{r}$.
Let $p: X'\to X$ be the global index one cover at point $x$ determined by $D$ and $X_0=X\backslash D$.
Let $X'_0=p^{-1}(X_0)$. Let
$\Delta'$ be the $\mathbb R$-divisor such that $K_{X'}+\Delta'=p^*(K_X+\Delta)$
and $\Delta'_0=\Delta'|_{X'_0}$.
Then $(X'_0,\Delta'_0)$ is also klt by Proposition \ref{index-prop} $(3)$ and
\cite[Proposition 5.20]{km}. Note that $(X', \Delta')$ is sub klt 
since $\Delta'$ may contain negative component supported on $p^{-1}(D)$.
Let $N'=p^*N$ be the ample $\mathbb R$-Cartier divisor on $X'$. Then 
$(N')^n=r\cdot N^n> n^n$.
 By \cite[Theorem 6.7.1]{ko} 
(where we only need to assume that $(X', \Delta')$
is klt in a neighborhood of $x'$, see also \cite[Proposition 3.3]{liu}), 
there is an effective $\mathbb R$-Cartier divisor 
$L'\sim_{\mathbb R} N'$ on $X'$ such that 
$(X', \Delta'+L')$ is not log canonical at $x'$.
Let $d_{x'}$ be the deficit of $0$ with respect to $(X', \Delta')$.

\begin{lem}\label{key-lem}
Notations are as above. If we further assume that $\ord_{x'} L'>  d_{x'}$, then
there is an effective $\mathbb R$-Cartier divisor $L \sim_{\mathbb R} N$, 
a positive number $0<c< 1$, and an 
open neighborhood $x\in U \subset X$ such that: 
\begin{itemize}
\item[$(1)$] $(U,(\Delta+cL)\vert_{U})$ is log canonical, and
\item[$(2)$] there is a minimal lc center $W$ of 
$(U,(\Delta+cL)\vert_{U})$ passing through $x$ with 
$\dim W<\dim X$.
\end{itemize}
\end{lem}

\begin{proof}
Let $L'_{g}$ be the Galois conjugates of $L'$ for $g\in G$.
Then 
 $\Sigma_g L'_{g}$ is $G$-invariant and thus,
there is an
effective $\mathbb R$-Cartier divisor 
$\widetilde {L}\sim_{\mathbb R} rN$ on $X$ such that $\Sigma_g L'_{g}=p^*\widetilde {L}$.
Note that
$$
\ord_{x'}\Sigma_g L'_{g}=r\cdot \ord_{x'} L'> r\cdot d_{x'}
$$
by Proposition \ref{index-prop} $(5)$ and assumption.
Therefore, there is a maximal number 
$$
c'\leq \frac{d_{x'}}{\ord_{x'}\Sigma_g L'_{g}} <\frac{1}{r}
$$ 
such that $(X'_0,\Delta'_0+c'\Sigma_g L'_{g})$ is log canonical at around point $x'$ 
by Remark \ref{rem-def-2}.
By shrinking $X_0$ and \cite[Proposition 5.20]{km}, 
$(U,(\Delta+c'\widetilde {L})\vert_{U})$ is log canonical in a suitable open neighborhood $x\in U$.
That is, there is an effective $\mathbb R$-Cartier divisor
$L=\frac{1}{r}\widetilde {L}\sim_{\mathbb R} N$
and a number $c=c'\cdot r< 1$ such that
$(U, (\Delta+cL)|_{U})$ is log canonical. This is $(1)$.
By our construction, $(2)$ is trivial.
\end{proof}

Without assuming that $\ord_{x'} L'> d_{x'}$,  
 it may happen that 
$x'\notin \Nklt (X', \Delta'+1/r(\Sigma_g L'_{g}))$ by replacing $L'$ with $1/r(\Sigma_g L'_{g})$,
and thus the inductive procedure stops.
Note also that $\ord_x L$ may be smaller than $d_x$  on $X$. 
It will be interesting to ask the relationship between $d_x$ and $d_{x'}$.

\section{Freeness for terminal singularities}\label{sec5}

It is known that
Fujita-type basepoint-freeness in dimension three has been proved up to  
Gorenstein terminal singularities by Lee in \cite{lee} and Kakimi in \cite{kkm} separately. 
When the threefold $X$ is not Gorenstein, 
the result is still not known. This is because that, for a given terminal point $x$,
$\mult_x X$ can go to infinity when the index $r$ is increasing
(cf. \cite[Theorem 2.1]{kkm2}). This makes the low bound of
$\ord_x L$ where $L\sim_{\mathbb R}N$ constructed by Riemann-Roch theorem
(as in the proof of \cite[Proposition 3.2]{h1}) so small that the inductive procedure stops.
In this section, we will overcome this problem by using global index one cover.
Note that  $\Sing X$ 
is a union of isolated points since $X$ has only terminal singularities.
For simplicity, we assume that $x:=\Sing X$ is a unique point.

\begin{thm}\label{thm-term}
Let $[X, \omega]$ be a projective quasi-log canonical pair such that 
$X$ is a normal $\mathbb Q$-factorial terminal threefold. 
Assume that $x:=\Sing X$ is a point and  $x\notin \Nqklt (X,\omega)$.
Let $r=\Index_xX$.
Let $M$ be a Cartier divisor on $X$. We put $N=M-\omega$ 
and assume that
$N^3 >\frac{27}{r}$ and that $N^k\cdot Z\geq 3^k$ for every subvariety 
$Z$ with $0<\dim Z=k<3$.
Then the complete linear system $|M|$ is basepoint-free at point $x$.
\end{thm}

\begin{proof}
By using Theorem \ref{u-thm} (see the proof of \cite[Theorem 3.2]{fujino-haidong-freeness}), 
we can take a 
boundary $\mathbb R$-divisor $\Delta_\varepsilon$ on 
$X$ such that $K_X+\Delta_\varepsilon \sim _{\mathbb R} \omega+\varepsilon N$ for 
$0<\varepsilon \ll 1$ and $\mathcal J(X, \Delta_\varepsilon)=\mathcal I_{\Nqklt(X, \omega)}$
where $\mathcal J(X, \Delta_\varepsilon)$ is the multiplier ideal sheaf of $(X, \Delta_\varepsilon)$. 
Since $\mathcal J(X, \Delta_\varepsilon)=\mathcal I_{\Nqklt(X, \omega)}$, $(X, \Delta_\varepsilon)$ is 
klt in a neighborhood of $x$. 
Let $p:X'\to X$ be the global index one cover at point $x$ determined by $D$ and $X_0=X\backslash D$.
Let $X'_0=p^{-1}(X_0)$. Let
$\Delta'_\varepsilon$ be the $\mathbb R$-divisor such that 
$K_{X'}+\Delta'_\varepsilon=p^*(K_X+\Delta_\varepsilon)$.
Then $(X'_0, \Delta'_\varepsilon|_{X'_0})$ is also klt by Proposition \ref{index-prop} $(3)$ and
\cite[Proposition 5.20]{km}. 
Let $N_\varepsilon:=M-(K_X+\Delta_\varepsilon)\sim_{\mathbb R}(1-\varepsilon)N$, $N':=p^*N$ and 
$N'_\varepsilon :=p^*N_\varepsilon\sim_{\mathbb R}(1-\varepsilon)N'$.
Note that
$(N')^3=r\cdot N^3>27$
and 
$$
(N')^k\cdot Z'\geq N^k\cdot p(Z')\geq 3^k
$$ 
for every subvariety 
$Z'\subset X'$ with that $0<\dim Z'=k<3$. 
Let $\sigma_3=3(1+3\varepsilon)$, $\sigma_2=\sigma_1=3(1-\varepsilon)$.
Then by choosing $\varepsilon$ small enough, we have that:
\begin{itemize}
\item[$(1)$] $(N'_\varepsilon)^3>\sigma^3_3>27$,
\item[$(2)$] $(N'_\varepsilon)^2 \cdot S' \geq \sigma^2_2$ for any irreducible surface $S'$,
\item[$(3)$] $N'_\varepsilon \cdot C' \geq \sigma_1$ for any irreducible curve $C'$,
\item[$(4)$] $\sigma_3>\sigma_2$,
\item[$(5)$] $\sigma_2(1-\frac{1}{\sigma_3})> 2-3\varepsilon$,
\item[$(6)$] $\sigma_1(1-\frac{1}{\sigma_3}-\frac{1}{\sigma_2})> 1-3\varepsilon$.
\end{itemize}

Since $X'$ has at most Gorenstein terminal singularity, $m_1=\mult_{x'}X'\leq 2$
(cf. \cite{ckm}).
By Riemann-Roch theorem, 
there is an effective $\mathbb R$-Cartier divisor 
$L'_3 \sim_{\mathbb R} N'_\varepsilon$ on $X'$ such that 
$\ord_{x'} L'_3 \geq \frac{\sigma_3}{\sqrt[3]{m_1}}$. 
Let $d_{x'}$ be the deficit of pair $(X'_0, \Delta'_\varepsilon|_{X'_0})$.
By Proposition \ref{prop-def} $(1)$,
\begin{equation}\label{eq5.1}
\begin{split}
 \ord_{x'} L'_3 \geq \sigma_3 >3\geq d_{x'} \quad \text{if $x'$ is smooth, or}
\\ \ord_{x'} L'_3 \geq \frac{\sigma_3}{\sqrt[3]{2}}>2
\geq d_{x'} \quad \text{if $x'$ is singular}. 
\end{split}
\end{equation}
Therefore, there is a maximal real number $0<c_3<1$ such that
$(X'_0, (\Delta'_\varepsilon+c_3L'_3)|_{X'_0})$ is log canonical at around point $x'$.
The same as Lemma \ref{key-lem}, we can assume that $L'_3$ is $G$-invariant by \eqref{eq5.1}
and thus $L'_3=p^*L_3$ where $(X, \Delta_\varepsilon+c_3L_3)$ is log canonical at around $x$.
Let $Z'$ be the closure of the minimal lc center of $(X'_0, (\Delta'_\varepsilon+c_3L'_3)|_{X'_0})$
in $X'$ and $Z$ be the minimal lc center of  $(X, \Delta_\varepsilon+c_3L_3)$ passing through $x$.
We discuss various cases according to the dimension of $Z'$.

\begin{case}\label{case5.1}
Assume that $\dim Z'=0$, that is, $x'=Z'$.
Then  $(X, \Delta_\varepsilon+c_3L_3)$ is log canonical at around $x$ and $x=p(Z')$
is the minimal lc center of $(X, \Delta_\varepsilon+c_3L_3)$ passing through $x$.
Let $W=\Nqlc(X, K_X+\Delta_\varepsilon+c_3L_3)\cup x$.
Since $M-(K_X+\Delta_\varepsilon+c_3L_3)\sim_{\mathbb R}(1-c_3)(1-\varepsilon)N$ is ample, 
the natural restriction map 
\begin{equation*}
H^0(X, \mathscr O_X(M))\to H^0(W, \mathscr O_W(M))
\end{equation*}
is surjective by the vanishing theorem. Since $x$ is an isolated point in $W$,
it is obviously that $|M|$ is basepoint-free at point $x$.
\end{case}

\begin{case}\label{case5.2}
Assume that $\dim Z'=1$, that is, $Z'$ is an irreducible curve smooth at $x'$.
By condition $(3)$, there is an effective $\mathbb R$-Cartier divisor 
$L'_{Z'} \sim_{\mathbb R} N'_\varepsilon|_{Z'} $ on $Z'$ such that $\ord_{x'} L'_{Z'} \geq \sigma_1$.
By Proposition \ref{prop-def} $(2)$, there is an effective $\mathbb R$-Cartier divisor 
$L'_1 \sim_{\mathbb R} N'_\varepsilon$ on $X'$ such that  
$\ord_{x'} L'_1  \geq \ord_{x'} L'_{Z'} \geq\sigma_1$.
By Proposition \ref{prop-def} $(1)$, 
\begin{equation}\label{eq5.2}
\ord_{x'} L'_1  \geq \sigma_1 =3(1-\varepsilon)> 1\geq d_{x'}(c_3L'_3).
\end{equation}
Therefore, there is a maximal real number $0<c_1<1$ such that
$(X'_0, (\Delta'_\varepsilon+c_3L'_3+c_1L'_1)|_{X'_0})$ is log canonical at around point $x'$.
By \eqref{eq5.2} and Lemma \ref{key-lem}, we assume that $L'_1$ is $G$-invariant 
and thus $L'_1=p^*L_1$ where $(X, \Delta_\varepsilon+c_3L_3+c_1L_1)$ is log canonical at around $x$.
In particular, $x$ is the minimal lc center of  $(X, \Delta_\varepsilon+c_3L_3+c_1L_1)$.
We need to show that $c_3+c_1<1$. If so, then 
$$
M-(K_X+\Delta_\varepsilon+c_3L_3+c_1L_1)\sim_{\mathbb R}(1-c_1-c_3)(1-\varepsilon)N
$$ 
is ample, and the natural restriction map 
\begin{equation*}
H^0(X, \mathscr O_X(M))\to H^0(W, \mathscr O_W(M))
\end{equation*}
is surjective by the vanishing theorem where $W=\Nqlc(X, K_X+\Delta+c_3L_3+c_1L_1) \cup x$.
Since $x$ is an isolated point in $W$,
it is obviously that $|M|$ is basepoint-free at point $x$.

\medskip

Therefore, we prove that $c_3+c_1<1$ in the rest of this case.
Let $b_3=\ord_{x'} L'_3 \geq \frac{\sigma_3}{\sqrt[3]{m_1}}$ 
and $b_1=\ord_{x'} L'_1 \geq \sigma_1$.
Let $d_0=d_{x'}$ and $d_1=d_{x'}(c_3L'_3)$. 
When $x'$ is smooth, we have the following relationship by Proposition \ref{prop-def}:
\begin{itemize}
\item[$(1)$] $d_0 \leq 3$,
\item[$(2)$] $c_3b_3 \leq d_0$,
\item[$(3)$] $d_1\leq d_0-c_3b_3$,
\item[$(4)$] $c_1b_1 \leq d_1$,
\end{itemize}
and thus
$$
c_3+c_1\leq c_3+\frac{d_0-c_3b_3}{b_1}
\leq c_3 +\frac{3-c_3\sigma_3}{\sigma_1}=\frac{3}{\sigma_1}+(1-\frac{\sigma_3}{\sigma_1})c_3
= \frac{1-4\varepsilon c_3}{1-\varepsilon}.
$$
Note that we can always assume that $c_3\geq \frac{2}{3}$. Otherwise,
we can easily check that on $(X, \Delta_\varepsilon+c_3L_3)$,
$M|_Z-(K_X+\Delta_\varepsilon+c_3L_3)|_Z$ 
satisfies Fujita's condition with respect to $\dim Z\leq 1$.
Then we can use induction on dimension and prove that $|M|_Z|$ is base-point free at point $x$
by \cite[Theorem 1.3]{fujino-haidong-freeness} and thus
$|M|$ is base-point free at point $x$ by the vanishing theorem.
Then 
$$
c_3+c_1 \leq \frac{1-4\varepsilon c_3}{1-\varepsilon}\leq \frac{3-8\varepsilon}{3-3\varepsilon}<1.
$$
When $x'$ is singular, we have the same relationship except that $d_0\leq 2$. Then
\begin{equation*}\label{eq5.3}
\begin{split}
c_3+c_1\leq c_3+\frac{d_0-c_3b_3}{b_1}
\leq c_3 +\frac{2\sqrt[3]{2}-c_3\sigma_3}{\sigma_1\sqrt[3]{2}}\\
=\frac{2}{\sigma_1}+(1-\frac{\sigma_3}{\sigma_1\sqrt[3]{2}})c_3
< \frac{2}{3-3\varepsilon}+\frac{1}{4}<1
\end{split}
\end{equation*}
and this is what we want.
\end{case}

\begin{case}\label{case5.3}
Assume that $\dim Z'=2$, that is, $Z'$ is an irreducible surface normal at $x'$.
Since $X'$ is Gorenstein and $Z'$ is $\mathbb Q$-Cartier by assumption,
$Z'$ is Cartier by \cite[Lemma 5.1]{kw2}.
Since $(X', Z')$ is plt at around $x'$,
$K_{Z'}=(K_{X'}+Z')|_{Z'}$ by adjunction and thus $Z'$ is also Gorenstein. 
In particular, $Z'$ has at most rational double point. Let $m_2=\mult_{x'} Z'$.
Then $m_2\leq 2$.
By condition $(2)$, there is an effective $\mathbb R$-Cartier divisor 
$L'_{Z'} \sim_{\mathbb R} N'_\varepsilon|_{Z'} $ on $Z'$ such that 
$\ord_{x'} L'_{Z'} \geq \frac{\sigma_2}{\sqrt{m_2}}$.
By a nice lifting, there is an effective $\mathbb R$-Cartier divisor 
$L'_2 \sim_{\mathbb R} N'_\varepsilon$ on $X'$ such that  
$b_2:=\ord_{x'} L'_2  \geq \ord_{x'} L'_{Z'} \geq \frac{\sigma_2}{\sqrt{m_2}}$.
Let $d_0=d_{x'}$ and $d_1=d_{x'}(c_3L'_3)$.
Let $b_3=\ord_{x'} L'_3 \geq \frac{\sigma_3}{\sqrt[3]{m_1}}$.
The same as Case \ref{case5.2}, we can assume that $c_3\geq \frac{1}{3}$.
By Proposition \ref{prop-def} and \cite[Lemma 3.3]{lee} (or \cite[Theorem 2.2]{kw},
where we only need to get a $L'_3$ such that $b_3>3+\sqrt{2}$ in the last case of the following conditions), 
\begin{equation}\label{eq5.3}
\begin{split}
\ord_{x'} L'_2 \geq 3- \varepsilon> 3(1-c_3)\geq d_0-c_3b_3 \geq d_1
\quad \text{if $m_1=1$ and $m_2$=1,}\\
\ord_{x'} L'_2 \geq 3- \varepsilon> 2\geq d_0-c_3b_3 \geq d_1
\quad \text{if $m_1=2$ and $m_2$=1,}\\
\ord_{x'} L'_2 \geq \frac{3- \varepsilon}{\sqrt{2}}> 2\geq d_0-c_3b_3 \geq d_1
\quad \text{if $m_1=2$ and $m_2$=2, }\\
\ord_{x'} L'_2 \geq \frac{3- \varepsilon}{\sqrt{2}}>3-\frac{3+\sqrt{2}}{3}\geq d_0-c_3b_3 \geq d_1
\quad \text{if $m_1=1$ and $m_2$=2.}
\end{split}
\end{equation}
Therefore, there is a maximal real number $0<c_2<1$ such that
$(X'_0, (\Delta'_\varepsilon+c_3L'_3+c_2L'_2)|_{X'_0})$ is log canonical at around point $x'$.
By \eqref{eq5.3} and Lemma \ref{key-lem}, we assume that $L'_2$ is $G$-invariant 
and thus $L'_2=p^*L_2$ where $(X, \Delta_\varepsilon+c_3L_3+c_2L_2)$ is log canonical at around $x$.
Let $S'$ be the closure of the minimal lc center of 
$(X'_0, (\Delta'_\varepsilon+c_3L'_3+c_2L'_2)|_{X'_0})$.
If $\dim S'=0$, then the same as Case \ref{case5.2}, we only need to show 
that $c_3+c_2<1$. 

\medskip

If $m_1=1$ and $m_1=1$, then
$$
c_3+c_2\leq c_3+\frac{d_0-c_3b_3}{b_2}
\leq c_3 +\frac{3-c_3\sigma_3}{\sigma_2}=\frac{3}{\sigma_2}+(1-\frac{\sigma_3}{\sigma_2})c_3
= \frac{1-4\varepsilon c_3}{1-\varepsilon}\leq \frac{3-4\varepsilon}{3-3\varepsilon}<1.
$$

If $m_1=2$ and $m_2=1$, then
$$
c_3+c_2 \leq c_3 +\frac{2\sqrt[3]{2}-c_3\sigma_3}{\sigma_2\sqrt[3]{2}}=
\frac{2}{\sigma_2}+(1-\frac{\sigma_3}{\sigma_2\sqrt[3]{2}})c_3
< \frac{2}{3-3\varepsilon}+\frac{1}{4}<1.
$$

If $m_1=2$ and $m_2=2$, then
$$
c_3+c_2 \leq c_3 +\frac{2\sqrt{2}\sqrt[3]{2}-\sqrt{2}c_3\sigma_3}{\sigma_2\sqrt[3]{2}}=
\frac{2\sqrt{2}}{\sigma_2}+(1-\frac{\sqrt{2}\sigma_3}{\sigma_2\sqrt[3]{2}})c_3
< \frac{2\sqrt{2}}{3-3\varepsilon}<1.
$$

Finally we consider the case $m_1=1$ and $m_2=2$. 
By choosing $\varepsilon$ small enough and \cite[Lemma 3.3]{lee}, we have 
\begin{equation}\label{eq5.4}
\frac{b_3}{b_3-2}<\frac{\sqrt{2}b_3}{b_3-1}<
\frac{(2-3\varepsilon)\sigma_3}{\sqrt{2}(\sigma_3-1)}<\frac{\sigma_2}{\sqrt{2}}\leq b_2.
\end{equation}
If $d_0-c_3b_3\geq 1$, then
$$
c_3+c_2 \leq c_3+\frac{d_0-c_3b_3}{b_2}=\frac{d_0}{b_2}+(1-\frac{b_3}{b_2})c_3
\leq \frac{1}{b_2}+\frac{d_0-1}{b_3}<\frac{b_3-2}{b_3}+\frac{d_0-1}{b_3}\leq 1;
$$
and if $d_0-c_3b_3<1$, then by \eqref{eq5.4} and simple calculation, we have that: 
$$
d_0-c_3b_3<(1-c_3)b_2.
$$
Therefore, 
$$
c_3+c_2 \leq c_3+\frac{d_0-c_3b_3}{b_2}<c_3+(1-c_3)<1.
$$
\end{case}

\begin{case}\label{case5.4}
We continue to discuss Case \ref{case5.3}. Assume that $\dim S'=1$.
That is, $S'$ is an irreducible curve smooth at $x'$.
By condition $(3)$, there is an effective $\mathbb R$-Cartier divisor 
$L'_{S'} \sim_{\mathbb R} N'_\varepsilon|_{S'} $ on $S'$ such that $\ord_{x'} L'_{S'} \geq \sigma_1$.
By a nice lifting, there is an effective $\mathbb R$-Cartier divisor 
$L'_1 \sim_{\mathbb R} N'_\varepsilon$ on $X'$ such that  
$b_1:=\ord_{x'} L'_1  \geq \ord_{x'} L'_{S'} \geq\sigma_1$.
Let $d_2=d_{x'}(c_3L'_3+c_2L'_2)$ and $b_2=\ord_{x'} L'_2$.
By Proposition \ref{prop-def}, 
\begin{equation}\label{eq5.5}
\ord_{x'} L'_1  \geq \sigma_1 =3(1-\varepsilon)> 1\geq d_2.
\end{equation}
Therefore, there is a maximal real number $0<c_1<1$ such that
$(X'_0, (\Delta'_\varepsilon+c_3L'_3+c_2L'_2+c_1L'_1)|_{X'_0})$ is log canonical at around point $x'$.
By \eqref{eq5.4} and Lemma \ref{key-lem}, we assume that $L'_1$ is $G$-invariant 
and thus $L'_1=p^*L_1$ where $(X, \Delta_\varepsilon+c_3L_3+c_2L_2+c_1L_1)$ is log canonical at around $x$.
In particular, $x$ is the minimal lc center of   $(X, \Delta_\varepsilon+c_3L_3+c_2L_2+c_1L_1)$.
Again, we only need to show that $c_3+c_2+c_1<1$.  

If $m_1=1$ and $m_2=1$, then
\begin{equation*}
\begin{split}
c_3+c_2+c_1 \leq c_3+c_2+\frac{d_2}{b_1}\leq c_3+c_2+\frac{d_0-c_3b_3-c_2b_2}{b_1}\\
\leq \frac{3}{\sigma_1}+(1-\frac{\sigma_3}{\sigma_1})c_3
= \frac{1-4\varepsilon c_3}{1-\varepsilon}\leq \frac{3-4\varepsilon}{3-3\varepsilon}<1.
\end{split}
\end{equation*}

If $m_1=2$ and $m_2=1$, then
\begin{equation*}
\begin{split}
c_3+c_2+c_1 \leq c_3+c_2+\frac{d_0-c_3b_3-c_2b_2}{b_1}\\
\leq \frac{2}{\sigma_1}+(1-\frac{\sigma_3}{\sigma_1\sqrt[3]{2}})c_3
< \frac{2}{\sigma_1}+\frac{1}{4}<1.
\end{split}
\end{equation*}

If $m_1=2$ and $m_2=2$, then 
$$
c_2\leq \frac{d_1}{b_2}\leq \frac{2-c_3b_3}{b_2},
$$ and
\begin{equation*}
\begin{split}
c_3+c_2+c_1 \leq c_3+c_2+\frac{d_0-c_3b_3-c_2b_2}{b_1}
\leq \frac{2}{\sigma_1}+(1-\frac{b_3}{\sigma_1})c_3
+(1-\frac{b_2}{\sigma_1})c_2\\
\leq  \frac{2}{\sigma_1}+(1-\frac{b_3}{\sigma_1})c_3
+(1-\frac{b_2}{\sigma_1})(\frac{2-c_3b_3}{b_2})=\frac{2}{b_2}+(1-\frac{b_3}{b_2})c_3.
\end{split}
\end{equation*}

If $m_1=1$ and $m_2=2$, then
$$
c_2\leq \frac{d_1}{b_2}\leq \frac{3-c_3b_3}{b_2}
$$ and
\begin{equation*}
\begin{split}
c_3+c_2+c_1 \leq c_3+c_2+\frac{d_0-c_3b_3-c_2b_2}{b_1}
\leq \frac{3}{\sigma_1}+(1-\frac{b_3}{\sigma_1})c_3
+(1-\frac{b_2}{\sigma_1})c_2\\
\leq  \frac{3}{\sigma_1}+(1-\frac{b_3}{\sigma_1})c_3
+(1-\frac{b_2}{\sigma_1})(\frac{3-c_3b_3}{b_2})
=\frac{3}{b_2}+(1-\frac{b_3}{b_2})c_3.
\end{split}
\end{equation*}

The last two inequations are exactly the same as the last two inequations in Case \ref{case3}, 
so we omit the tedious calculations.
\end{case}
Anyway, we finish our proof.
\end{proof}

\begin{rem}\label{rem-term}
Comparing the six conditions at the beginning of above proof 
with that in \cite[Corollary 1.6]{lee}, we see that one of the troubles
to use \cite[Corollary 1.6]{lee} directly
is the small perturbation of $3\varepsilon$ in $(5)$ and $(6)$.
Essentially, we can weaken Fujita's condition by a sufficiently small perturbation
in lower dimensions such as $N^k\cdot Z\geq (\dim X-\varepsilon)^k$ for any subvariety
 with $\dim Z=k<\dim X$. 
This is because in the condition that
$N^{\dim X}>(\dim X)^{\dim X}$, the difference $N^{\dim X}-(\dim X)^{\dim X}$
will give enough room to offset the negative effect of $\varepsilon$.
We already saw this spirit in the proof of \cite[Theorem 3.2]{fujino-haidong-freeness}.
\end{rem}

By Theorem \ref{thm-term}, we can see that the bigger $r$ is, the weaker conditions than
Fujita's condition we need. In particular,
we immediately have the following corollary by Theorem \ref{thm-term}:

\begin{cor}\label{cor-term}
Let $[X, \omega]$ be a projective quasi-log canonical pair such that 
$X$ is a normal $\mathbb Q$-factorial terminal threefold. 
Assume that $x:=\Sing X$ is a point and $x\notin \Nqklt (X,\omega)$. 
Let $M$ be a Cartier divisor on $X$. We put $N=M-\omega$ 
and assume that $N$ satisfies Fujita's condition. 
Then there exists a section $s\in H^0(X, \mathcal I_{\Nqklt(X, \omega)}
\otimes \mathcal O_X(M))$ such that 
$s(x)\ne 0$, where $\mathcal I_{\Nqklt(X, \omega)}$ is 
the defining ideal sheaf of $\Nqklt(X, \omega)$ on $X$. 
\end{cor}

\begin{proof}
In the proof of Theorem \ref{thm-term},
we in fact get an induced quasi-log structure
$$
[X, K_X+\Delta_\varepsilon+c_3L_3+c_2L_2+c_1L_1]
$$
such that $x$ is the minimal qlc center of this quasi-log structure
and $c_3+c_2+c_1<1$ where $c_i$ may be zero for some $i$.
Let $V$ be the union of all irreducible qlc centers of 
$[X, K_X+\Delta_\varepsilon+c_3L_3+c_2L_2+c_1L_1]$ passing through $x$.
Let $W$ be the closure of 
$$
\Nqklt(X, K_X+\Delta_\varepsilon+c_3L_3+c_2L_2+c_1L_1)\backslash V.
$$ Note that $W$ is a union of some qlc centers and $\Nqlc (X, K_X+\Delta_\varepsilon+c_3L_3+c_2L_2+c_1L_1)$. 
Then:
$$
M-(K_X+\Delta_\varepsilon+c_3L_3+c_2L_2+c_1L_1)\sim_{\mathbb R}(1-c_1-c_2-c_3)(1-\varepsilon)N
$$ 
is ample, and the natural restriction map 
\begin{equation*}
H^0(X, \mathscr O_X(M))\to H^0(W\cup x, \mathscr O_{W\cup x}(M))
\end{equation*}
is surjective by the vanishing theorem.
Since $x$ is an isolated point in $W\cup x$, there exists a section $s\in H^0(X, \mathcal O_X(M))$ such that 
$s(x)\ne 0$ and $s(W)=0$. 
Since every $L_i$ is effective, we have that
$$
\Nqklt(X, \omega)=
\Nqklt(X, K_X+\Delta_\varepsilon) \subset W
$$
by the construction of $W$. Therefore, $s\in H^0(X, \mathcal I_{\Nqklt(X, \omega)}
\otimes \mathcal O_X(M))$ and $s(x)\ne 0$.
\end{proof}

\setcounter{case}{0}
\section{Freeness for qlc singularities}\label{sec6}

First we deal with the normal case.

\begin{thm}\label{thm-norm}
Let $[X, \omega]$ be a projective quasi-log canonical pair such that 
$X$ is a normal threefold. 
Let $x$ be a closed point on $X$.
Let $M$ be a Cartier divisor on $X$. We put $N=M-\omega$ 
and assume that $N$ satisfies Fujita's condition. 
Then the complete linear system $|M|$ is basepoint-free at point $x$.
\end{thm}

\begin{proof}
Assume that $x\in \Nqklt (X,\omega)$. 
Let $W$ be the minimal irreducible qlc center passing through $x$.
Then by adjunction theorem, $[W, \omega|_W]$ is an induced qlc pair. 
Note that $\dim W\leq 2$. Then
it is easy to check that $N|_W=M|_W-\omega|_W$ also satisfies Fujita's condition 
with respect to $\dim W$. Therefore, $|M|_W|$ is basepoint-free at point $x$
by \cite{fujino-haidong-freeness}. 
By the vanishing theorem, the natural restriction map 
\begin{equation*}
H^0(X, \mathscr O_X(M))\to H^0(W, \mathscr O_W(M))
\end{equation*}
is surjective. Therefore, $|M|$ is also basepoint-free at point $x$.

\medskip

Then we assume that
$x\notin \Nqklt (X,\omega)$. Let $h: \widetilde{X} \to X$ be the $\mathbb Q$-factorial terminalization 
and $[\widetilde{X}, \widetilde{\omega}]$  be the induced qlc pair by Lemma \ref{term-lem}.
In particular, $\widetilde{X}$ is $\mathbb Q$-factorial and terminal.
By Corollary \ref{cor-term}, we can assume that $\dim h^{-1}(x)\geq 1$. Since $\dim \Sing \widetilde{X}=0$, 
we can choose a general point $\widetilde{x} \in h^{-1}(x) $
which is smooth on $\widetilde{X}$. Let $\widetilde{M}=h^*M$ and $\widetilde{N}=h^*N$.
By using Theorem \ref{u-thm}, we can take a 
boundary $\mathbb R$-divisor $\Delta_\varepsilon$ on 
$X$ with $K_X+\Delta_\varepsilon \sim _{\mathbb R} \omega+\varepsilon N$ for 
$0<\varepsilon \ll 1$ such that  $(X, \Delta_\varepsilon)$ is 
klt in a neighborhood of $x$,
and a boundary $\mathbb R$-divisor $\widetilde{\Delta}_\varepsilon$ on 
$ \widetilde{X} $ with
$K_{\widetilde{X}}+\widetilde{\Delta}_\varepsilon=
h^*(K_X+\Delta_\varepsilon)\sim _{\mathbb R} \widetilde{\omega}+\varepsilon  \widetilde{N}$ 
such that $(\widetilde{X}, \widetilde{\Delta}_\varepsilon)$ is 
klt in a neighborhood of $\widetilde{x}$.
Let $N_\varepsilon=M-(K_X+\Delta_\varepsilon)\sim_{\mathbb R} (1-\varepsilon)N$ and
$\widetilde{N}_\varepsilon=h^*N_\varepsilon\sim_{\mathbb R}(1-\varepsilon)\widetilde{N}$.
Let $\sigma_3=3(1+\varepsilon)$, $\sigma_2=\sigma_1=3(1-\varepsilon)$. Then
by choosing $\varepsilon$ small enough, we have that:
\begin{itemize}
\item[$(1)$] $\widetilde{N}_\varepsilon^3=N_\varepsilon^3>\sigma^3_3>27$,
\item[$(2)$] $\widetilde{N}_\varepsilon^2 \cdot S =N_\varepsilon^2 \cdot  h(S)\geq \sigma^2_2$ 
for any irreducible surface $S$ such that $\dim h(S)=2$,
\item[$(3)$] $\widetilde{N}_\varepsilon \cdot C=N_\varepsilon \cdot h(C)\geq \sigma_1$ 
for any irreducible curve $C$ such that $\dim h(C)=1$.
\end{itemize}
By condition $(1)$, we can find an effective $\mathbb R$-Cartier divisor 
$L_3 \sim_{\mathbb R} N_\varepsilon$, 
$\widetilde{L}_3=h^*L_3\sim_{\mathbb R} \widetilde{N}_\varepsilon$,  and
a maximal positive number $0<c_3< 1$ such that 
$b_3:=\ord_{\widetilde{x}}\widetilde{L}_3 > \sigma_3$,
$(\widetilde{X}, \widetilde{\Delta}_\varepsilon+\widetilde{L}_3)$ is not lc at point $\widetilde{x}$ 
but $(\widetilde{X}, \widetilde{\Delta}_\varepsilon+c_3 \widetilde{L}_3)$ is lc at point $\widetilde{x}$.
Let $\widetilde Z$ be the irreducible minimal lc center passing through $\widetilde{x}$. Then
$\dim \widetilde Z\leq 2$ and $\widetilde Z$ is normal at around point $\widetilde{x}$.
Note that $\widetilde Z$ may be contained in the exceptional locus of $h$.
We discuss various cases according to the dimension of $\widetilde Z$.

\begin{case}\label{case1}
Assume that $\dim \widetilde Z>\dim h(\widetilde Z)$.
If $\dim h(\widetilde Z)=0$, that is, $x=h(\widetilde Z)$ since 
$\widetilde  Z$ passing through $\widetilde{x}\in h^{-1}(x)$,
then by Corollary \ref{term-cor}, $x$  is the minimal lc center of $(X, \Delta_\varepsilon+c_3L_3)$.
Note that we get a natural quasi-log structure on $[X, K_X+\Delta_{\varepsilon}+c_3L_3]$. 
That is, $x$ is the minimal qlc center of $[X, K_X+\Delta_{\varepsilon}+c_3L_3]$
since  $x$ is the minimal lc center of 
$(X, \Delta_\varepsilon+c_3L_3)$. 
Let $W=\Nqlc(X, K_X+\Delta_{\varepsilon}+c_3L_3)\cup x$.
Since $M-(K_X+\Delta_\varepsilon+c_3L_3)\sim_{\mathbb R}(1-c_3)(1-\varepsilon)N$ is ample, 
the natural restriction map 
\begin{equation*}
H^0(X, \mathscr O_X(M))\to H^0(W, \mathscr O_W(M))
\end{equation*}
is surjective by the vanishing theorem. Since $x$ is an isolated point in $W$,
it is obviously that $|M|$ is basepoint-free at point $x$.

\medskip

The left case for $\dim \widetilde Z>\dim h(\widetilde Z)$ 
is that $\dim \widetilde Z=2$ and $\dim h(\widetilde Z)=1$. Let $C=h(\widetilde Z)$ be the curve on $X$.
Note that $C$ is the minimal lc center of $(X, \Delta_\varepsilon+c_3L_3)$ by Corollary \ref{term-cor}.
By condition $(3)$, there is an effective $\mathbb R$-Cartier divisor 
$L_C \sim_{\mathbb R} N_\varepsilon|_C$ on $C$ such that $\ord_x L_C \geq \sigma_1>1$
and $[C, (K_X+\Delta_\varepsilon+c_3L_3)|_C+L_C]$ is not qlc at point $x$.
By a nice lifting, there is an effective $\mathbb R$-Cartier divisor 
$L_1 \sim_{\mathbb R} N_\varepsilon$ on $X$ such that $\ord_x L_1 \geq \sigma_1$
and $(X, \Delta_\varepsilon+c_3L_3+L_1)$ is not lc at point $x$.
Therefore, there is a maximal number $0<c_1<1$ such that $(X, \Delta_\varepsilon+c_3L_3+c_1L_1)$ is lc
at point $x$ and $x$ is exactly the minimal lc center of  $(X, \Delta_\varepsilon+c_3L_3+c_1L_1)$.
We need to show that $c_3+c_1<1$. Then
$$
M-(K_X+\Delta_\varepsilon+c_3L_3+c_1L_1)\sim_{\mathbb R}(1-c_3-c_1)(1-\varepsilon)N
$$ 
is ample, and the natural restriction map 
\begin{equation*}
H^0(X, \mathscr O_X(M))\to H^0(W, \mathscr O_W(M))
\end{equation*}
is surjective by the vanishing theorem where 
$W=\Nqlc(X, K_X+\Delta_{\varepsilon}+c_3L_3+c_1L_1)\cup x$.
Since $x$ is an isolated point in $W$,
it is obviously that $|M|$ is basepoint-free at point $x$.

\medskip

Therefore, we prove that $c_3+c_1<1$ in the rest of this case.
Let $\widetilde{L}_1=h^*L_1$. Then by Corollary \ref{term-cor},
$(\widetilde{X}, \widetilde{\Delta}_\varepsilon+c_3 \widetilde{L}_3+c_1\widetilde{L}_1)$ 
is lc at point $\widetilde{x}$.
Note that
$\ord_{h^{-1}(x)} \widetilde{L}_1=\ord_x L_1\geq \sigma_1$.
Since $\mathcal I^k_{h^{-1}(x)} \to \mathcal I^k_{\widetilde{x}}$ is injective
for any positive integer $k$,
$$
\ord_{\widetilde{x}} \widetilde{L}_1 \geq \ord_{h^{-1}(x)} \widetilde{L}_1\geq \sigma_1.
$$
Let $b_3=\ord_{\widetilde{x}} \widetilde{L}_3> \sigma_3$ 
and $b_1=\ord_{\widetilde{x}} \widetilde{L}_1 \geq \sigma_1$.
Let $d_0=d_{\widetilde{x}}$ and $d_1=d_{\widetilde{x}}(c_3L_3)$. Then by Proposition \ref{prop-def}:
\begin{itemize}
\item[$(1)$] $d_0 \leq 3$,
\item[$(2)$] $c_3b_3 \leq d_0$,
\item[$(3)$] $d_1\leq d_0-c_3b_3$, and
\item[$(4)$] $c_1b_1 \leq d_1$,
\end{itemize}
and thus
$$
c_3+c_1\leq c_3+\frac{d_0-c_3b_3}{b_1}
\leq \frac{3}{\sigma_1}+(1-\frac{\sigma_3}{\sigma_1})c_3
= \frac{1-2\varepsilon c_3}{1-\varepsilon}
\leq \frac{3-4\varepsilon}{3-3\varepsilon}<1
$$
by adding assumption that $c_3\geq \frac{2}{3}$ 
as in Case \ref{case5.2} of Theorem \ref{thm-term}.
\end{case}

\begin{case}\label{case2}
Assume that $\dim Z=\dim h(Z)=2$. Then by condition $(2)$ 
we can find an effective $\mathbb R$-Cartier divisor 
$L_2 \sim_{\mathbb R} N_\varepsilon$, 
$\widetilde{L}_2=h^*L_2\sim_{\mathbb R} \widetilde{N}_\varepsilon$,  and
a maximal positive number $0<c_2< 1$ such that 
$b_2:=\ord_{\widetilde{x}}\widetilde{L}_2 \geq \sigma_2$,
$(\widetilde{X}, \widetilde{\Delta}_\varepsilon+c_3\widetilde{L}_3+\widetilde{L}_2)$ 
is not lc at point $\widetilde{x}$ 
but $(\widetilde{X}, \widetilde{\Delta}_\varepsilon+c_3\widetilde{L}_3+c_2\widetilde{L}_2)$ 
is lc at point $\widetilde{x}$.
Let $\widetilde T$ be the irreducible minimal lc center passing through $\widetilde{x}$. Then
$\dim \widetilde T\leq 1$ and $\widetilde T$ is normal at around point $\widetilde{x}$.
We further assume in this case that $\dim \widetilde T > \dim h(\widetilde T)$. 
That is, $x=h(\widetilde T)$.

\medskip

Under these assumptions, we only need to prove that $c_3+c_2<1$ and the rest are the same
as Case \ref{case1}. By \cite[Theorem 2.2]{kw} or \cite[Lemma 3.3]{lee},
we can further assume that $b_3=\ord_{\widetilde{x}} \widetilde{L}_3> 3+\sqrt{2}$.
Then we can get $c_3+c_2<1$ exactly the same as Case \ref{case5.3} of Theorem \ref{thm-term}.
\end{case}

\begin{case}\label{case3}
Besides above cases, we have that $\dim \widetilde Z=\dim h(\widetilde Z)$ 
and $\dim \widetilde T=\dim h(\widetilde T)$ 
(when $\dim \widetilde Z=\dim h(\widetilde Z)=2$). 
Then condition $(2)$ and $(3)$ 
make sure that we could create the inductive procedure 
exactly the same as the proof of Theorem \ref{thm-term}.
\end{case}
Anyway, we finish our proof of Fujita-type freeness for normal qlc threefolds.
\end{proof}

\begin{rem}\label{rem6.2}
As we saw in above proof, 
to create the desired inductive procedure 
we only need the Fujita's condition holds true for those possible lc (or qlc)
minimal centers.
\end{rem}

Finally, we finish our proof of Fujita-type basepoint-freeness for general quasi-log canonical threefolds.

\begin{thm}\label{thm-general}
Let $[X, \omega]$ be a  projective quasi-log canonical threefold.
Let $x$ be a closed point on $X$.
Let $M$ be a Cartier divisor on $X$. We put $N=M-\omega$ 
and assume that $N$ satisfies Fujita's condition. 
Then the complete linear system $|M|$ is basepoint-free at $x$.
\end{thm}

\begin{proof}
Let $x$ be an arbitrary closed point of $X$ and
let $W$ be the irreducible minimal qlc stratum of $[X, \omega]$ 
passing through $x$.
By adjunction theorem, $[W, \omega|_W]$ is a quasi-log canonical pair. By 
the vanishing theorem, the natural restriction map 
\begin{equation}\label{eq6.1}
H^0(X, \mathscr O_X(M))\to H^0(W, \mathscr O_W(M))
\end{equation}
is surjective. 
From now on, we will see that $|M|$ is basepoint-free in a neighborhood of
$x$. If $W=x$, that is, $x$ is a qlc center of $[X, \omega]$, 
then the complete linear system $|M|$ is
obviously basepoint-free in a neighborhood of $x$ by the surjection \eqref{eq6.1}. 
Let us consider the case where $0<m=\dim W<3$.
Let $M_W=M|_W$ and $N_W=N|_W=(M-\omega)|_W$.
Then $N_W^m\cdot W=N^m\cdot W \geq 3^m>m^m$
and $N_W^k\cdot Z=N^k\cdot Z \geq 3^k\geq m^k$ for every subvariety 
$Z\subset W$ with $0<\dim Z=k<m$. That is,
$N_W$ also satisfies Fujita's condition.
Using induction of dimension, $|M_W|$ is basepoint-free at $x$ by \cite{fujino-haidong-freeness}.
Combining \eqref{eq6.1}, we see that $|M|$ is basepoint-free at $x$. 
Thus we may assume that $\dim W=\dim X=n$.
It is also easy to check that $N_W=N|_W$ satisfies Fujita's condition as above.
Therefore, by replacing $X$ with $W$, we can assume that $X$ is irreducible 
and $x\notin \Nqklt(X,\omega)$. In particular, $X$ is normal near point $x$.

\medskip

Let $\nu: \widetilde{X}\to X$ be the normalization. 
Note that $[\widetilde{X}, \nu^*\omega]$ is a qlc pair by Lemma \ref{lem-normal}.
We put $\widetilde{M}=\nu^*M$ 
and $\widetilde{N}=\nu^*N=\widetilde{M}-\nu^*\omega$.
It is obvious that  $\widetilde{M}$ is Cartier.
Moreover, $(\widetilde{N})^3 \cdot \widetilde{X}=N^3 \cdot X> 27$
and $(\widetilde{N})^k\cdot Z\geq N^k \cdot \nu(Z)\geq 3^k$ for every subvariety 
$Z\subset \widetilde{X}$ with $0<\dim Z=k<3$.
Note that $\dim \nu(Z)=\dim Z=k$ 
since normalization $\nu$ is finite.
We also note that, $\widetilde x:=\nu^{-1}(x)$ is a point
since $\nu: \widetilde{X}\to X$ is an isomorphism over some open
neighborhood of the normal point $x$,
and that the non-normal part of $X$ is contained in  $\Nqklt(X, \omega)$ by 
Lemma \ref{lem-normal}. 
The same as Corollary \ref{cor-term}, there is a section 
$$
\widetilde{s} \in H^0(\widetilde{X}, \mathcal I_{\Nqklt(\widetilde{X}, \nu^*\omega)}
\otimes \mathcal \mathscr O_{\widetilde{X}}(\widetilde{M})).
$$ 
such that $\widetilde{s}(\widetilde x)\neq 0$ by the proof of Theorem \ref{thm-norm}.
By $\nu_*\mathcal I_{\Nqklt(\widetilde{X}, \nu^*\omega)}=\mathcal I_{\Nqklt(X, \omega)}$
in Lemma \ref{lem-normal}, we have that:
$$
H^0(\widetilde{X}, \mathcal I_{\Nqklt(\widetilde{X}, \nu^*\omega)
}\otimes \mathscr O_{\widetilde{X}}(\widetilde{M})) \cong 
H^0(X,\mathcal I_{\Nqklt(X, \omega)} \otimes 
\mathscr O_{X}(M)).
$$ 
Thus we can descend the section $\widetilde{s}$ to a section
$s\in H^0(X, \mathcal I_{\Nqklt(X, \omega)}
\otimes \mathcal O_X(M))$ and $s(x)\neq 0$.
This $s \in H^0(X, \mathscr O_{X}(M))$ is what we want.
\end{proof}

\section{Appendix}\label{sec7}
Let  $\left(X, \omega, f:(Y, B_Y)\to X\right)$ be a quasi-log canonical pair such that 
$X$ is an $n$-dimensional normal variety and 
$x$ be a closed point such that $x\notin \Nqklt(X,\omega)$.
By the universal property of blowing up, 
we can assume that $f$ factors through the blowing up $\varphi$ defined in Definition \ref{defn-ord} by
morphism $g:Y\to X'$ such that $f=\varphi\circ g$.
That is, there is a commutative diagram:
$$
\xymatrix{
& Y\ar[ld]_-{g}\ar[d]^-f\\ 
X' \ar[r]_-{\varphi}& X.
}
$$
Let $G$ be an effective $\mathbb R$-Cartier divisor 
and $[X,\omega+G]$ be the induced quasi-log structure
by \cite[Lemma 4.6]{fujino-slc-surface}.

\begin{defn}\label{defn-def-qlc}
Let $B_Y=\sum a_iF_i$, $f^*G=\sum b_iF_i$ and
$g^*E=\sum e_iF_i$ where $F_i$ are simple normal crossing divisors. 
If $x\notin\Nqklt(X,\omega+G)$, then the {\em{deficit of $G$}} at $x$, is defined as 
$$
d_x[X,\omega+G]=\inf_{f(F_i)=x}\{\frac{1-a_i-b_i}{e_i}\},
$$
where $f$ varies among all quasi-log resolutions
factoring through $\varphi$.
If $x\in \Nqklt(X,\omega+G)$ but $x\notin \Nqklt(X,\omega+(1-t)G)$ for any $0<t<1$,
then the {\em{deficit of $G$}} at $x$, is defined as 
$$
d_x[X,\omega+G]=\lim_{t\to 0^{+}} d_x[X,\omega+(1-t)G].
$$
\end{defn}

\begin{lem}\label{lem-def-qlc-1}
It is equivalent to define {\em{deficit of $G$}} as the smallest $c\in \mathbb R_{\geq 0}$ such that 
for any effective $\mathbb R$-Cartier divisor $D$ with $\ord_x D\geq c$, 
we have $x\in\Nqklt (X,\omega+G+D)$.
\end{lem}

\begin{proof}
Let $d=d_x[X,\omega+G]$. Let $0<\varepsilon\ll 1$ be a sufficiently small number.
Let $D$ be  any given effective $\mathbb R$-Cartier divisor with $\ord_x D \geq d$.
Then 
\begin{equation}\label{eq7.1}
g_i:=\mult_{F_i} f^*D =\ord_x D\cdot e_i \geq d\cdot e_i.
\end{equation}
By Definition \ref{defn-def-qlc},
there is a prime divisor $F$ on $Y$ such that $f(F)=x$ and 
\begin{equation}\label{eq7.2}
1-a-b<(d+\varepsilon)\cdot e
\end{equation}
where $a=\mult_F B_Y$,  $b=\mult_F f^*G$ and $e=\mult_F g^*E$.
 Let 
$$
f: (Y, B_Y+f^*G+f^*D) \to [X, \omega+G+D]
$$
be the induced quasi-log structure. Then \eqref{eq7.1} and \eqref{eq7.2} imply that
$$
\mult_F(B_Y+f^*G+f^*D)=a+b+g \geq a+b+d\cdot e>1-\varepsilon\cdot e.
$$
where $g=\mult_F f^*D$.
Note that the numbers $a$, $b$, $g$ and $e$ will not change anymore if we further blow up $(Y, B_Y)$
and consider the strict transform of $F$. Therefore, let $\varepsilon \to 0$, we have that
(after replacing $F$ with its strict tranform):
$$
\mult_F(B_Y+f^*G+f^*D) \geq 1.
$$
That is, $x\in\Nqklt (X,\omega+G+D)$ which implies that $c\leq d$ by assumption.

\medskip

Conversely, let $D$ be an effective $\mathbb R$-Cartier divisor with 
$\varepsilon +c > \ord_x D \geq c$.
Then by assumption, there is a prime divisor $F$ on $Y$ such that $f(F)=x$ and 
$$
a+b +(\varepsilon +c)\cdot e > \mult_F(B_Y+f^*G+f^*D)=a+b+g \geq 1.
$$
Therefore, $\varepsilon +c > d$ by definition of $d$. Let  $\varepsilon \to 0$, we have
that $c\geq d$. We get what we want.
\end{proof}

Let $N$ be an ample $\mathbb R$-Cartier divisor.
By using Theorem \ref{u-thm} (see the proof of \cite[Theorem 3.2]{fujino-haidong-freeness}), 
we can take a 
boundary $\mathbb R$-divisor $\Delta_\varepsilon$ on 
$X$ such that $K_X+\Delta_\varepsilon \sim _{\mathbb R} \omega+\varepsilon N$ for 
$0<\varepsilon \ll 1$ and $\mathcal J(X, \Delta_\varepsilon)=\mathcal I_{\Nqklt(X, \omega)}$
where $\mathcal J(X, \Delta_\varepsilon)$ is the multiplier ideal sheaf of $(X, \Delta_\varepsilon)$. 
Since $\mathcal J(X, \Delta_\varepsilon)=\mathcal I_{\Nqklt(X, \omega)}$, $(X, \Delta_\varepsilon)$ is 
klt in a neighborhood of $x$. 
Let $D_\varepsilon$ be an effective $\mathbb R$-Cartier divisor.
Note that we get a natural quasi-log structure on $[X, \omega_{\varepsilon}+D_\varepsilon]$ with 
$\omega_\varepsilon:=K_X+\Delta_{\varepsilon}$. 
$W_\varepsilon$ is the minimal qlc center of $[X, \omega_{\varepsilon}+D_\varepsilon]$
passing through $x$,
is equivalent to say that, $W_\varepsilon$ is the minimal log canonical center 
of $(X, \Delta_{\varepsilon}+D_\varepsilon)$ passing through $x$. 
Let $d_\varepsilon$ be the deficit of $0$ with respect to $(X, \Delta_\varepsilon)$ defined in
Definition \ref{defn-def}. Then:

\begin{lem}\label{lem-def-qlc-2}
$d_x[X,\omega]=\lim_{\varepsilon \to 0^{+}} d_\varepsilon$.
\end{lem}

\begin{proof}
By blowing up $(Y, B_Y)$ further and the universal property of blowing up,  
there is a commutative diagram as follows:
$$
\xymatrix{
Z\ar[d]_-h\ar[dr]_-{p} & Y\ar[d]^-f\ar[l]_-g\\ 
X' \ar[r]_-{\varphi}& X.
}
$$
where $\varphi$ is the blowing up of point $x$ 
as in Definition \ref{defn-def} and Definition \ref{defn-def-qlc},
and $p$ is a sufficiently high log resolution as in Theorem \ref{u-thm}.
Let $D_\varepsilon$ be an effective $\mathbb R$-Cartier divisor such that 
$\ord_x D_\varepsilon\geq d_\varepsilon$.
By Lemma \ref{lem-def-qlc-1}, $x \in \Nqklt (X, \omega_{\varepsilon}+D_\varepsilon)$.
By our construction, $\Nqklt (X, \omega_{\varepsilon}+D_\varepsilon)$ is dominated by 
$(B_Y+\varepsilon f^*N+f^*D_\varepsilon)^{\geq 1}$. 
Since $$
(B_Y+f^*D_\varepsilon)^{\geq 1} \preceq (B_Y+\varepsilon f^*N+f^*D_\varepsilon)^{\geq 1},
$$
$x$ is not necessarily contained in $\Nqklt (X, \omega+ D_\varepsilon)$.
This means that
$d_x[X,\omega] \geq d_\varepsilon$. 
By limiting, $d_x[X,\omega]\geq \lim_{\varepsilon \to 0^{+}} d_\varepsilon$.

\medskip

Conversely, let $t$ be a sufficiently small number.
Let $D_t$ be an effective $\mathbb R$-Cartier divisor such that 
$\ord_x D_t= d_x[X,\omega]-t$ and $x\notin \Nqklt (\omega+D_t)$. 
We can choose a small number $0<\varepsilon \ll 1$ such that
$\varepsilon \cdot \ord_x N<t$ and $x\notin \Nqklt (\omega+\varepsilon N+D_t)$. 
That is, $d_\varepsilon \geq d_x[X,\omega]-t$.
Let $\varepsilon \to 0$, we have that  $\lim_{\varepsilon \to 0^{+}} d_\varepsilon \geq d_x[X,\omega]-t$.
Let $t\to 0$, we have that $\lim_{\varepsilon \to 0^{+}} d_\varepsilon \geq d_x[X,\omega]$.
We get what we want.
\end{proof}

By Lemma \ref{lem-def-qlc-2}, any property belonging to deficit of klt pair is also belonging to 
deficit of qlc pair. In particular, since a klt pair $(X,\Delta)$ has a natural qlc structure,
the deficits of both cases are the same.
Note that if \cite[Conjecture 1.5]{fujino-haidong} is true, then 
Lemma \ref{lem-def-qlc-2} is tedious.


\end{document}